\documentclass{article}
\usepackage{amsmath, amssymb, amsthm}
\usepackage{bm} 
\usepackage{mathptmx} 
\usepackage{graphicx}
\usepackage[colorlinks=false, pdfborder={0 0 1}]{hyperref}
\usepackage[a4paper, top=3cm, bottom=3cm, left=2cm, right=2cm]{geometry}
\newtheorem{lemma}{Lemma}
\newtheorem{proposition}{Proposition}
\newtheorem{theorem}{Theorem}
\newtheorem{problem}{Problem}
\newtheorem{definition}{Definition}
\newcommand{\problemname}{P\textsuperscript{w}}
\newcommand{\problemref}[1]{\problemname}

\title{On Existence and Uniqueness of the Solution of a Two-Surfaces Contact Problem Using a Fixed Point Approach}

\author{Abdelkrim Atailia\thanks{Email: \texttt{abdelkrim.atailia@univ-annaba.dz},\texttt{ atailia.karim.001@gmail.com}} \and Frekh Taallah\thanks{Email: \texttt{frekh.taallah@univ-annaba.dz}}}

\date{Laboratory of Mathematical Modeling and Numerical Simulation (LAM²SIN), University of Badji-Mokhtar Annaba, Annaba, Algeria}

\begin{document}
	
	\maketitle

	\begin{abstract} In this work, we give the proof of the existence and uniqueness of the solution to the weak form of a two-surfaces contact problem using fixed point approach. We begin by modeling the evolution of a two deformable surfaces contact problem with a general viscoplastic law, the contact is considered frictionless and governed by the Signorini-type condition with an initial gap. Then, we derive the variational formulation of the classical problem. Finally, we conclude our work by establishing an existence and uniqueness theorem for the weak form.   
		
	\end{abstract}

	\section{Introduction}
	The contact between two surfaces is a phenomenon that occurs at each moment in the physical existence, which can lead to a deformation of at least one surface of two surfaces,depending on the nature of materials that comes to contact. This phenomenon generates a myriad of models depending on the type of material, the applied forces, etc. Mathematically we are able to study only few classes of contact problems due to the complexities that arise with them. In our work, we investigate the existence and uniqueness of a weak solution to a specific class of contact problems, which has never been studied before. We used the fixed point approach rather than the classical method. The advantage of studying this class by using fixed point approach is to set the base for studying its stability in the sense of Tykhonov, and this in turn, allow us to give the stability of a broader class of contact problems. In this classes the stability is not achieved by the classical approach.  Through the Tykhonov method, less conditions on the initial data are enforced. Hence, we are at the point where we are able to predict  the deformation of a boarder class of contact phenomenon and implementing them numerically in engineering and in the simulated world whenever possible.  
	
	The majority of previous works that used the fixed point approach, focused on investigating the cases where a body comes to contact with a rigid fixed obstacle (foundation), as in \textbf{Figure}\ref{figure1}. Less works supposed that the foundation is deformable and considered the body to be rectangular like in \cite{Sofonea2017}. Rare works focused on the contact of two deformable bodies.  Like in \cite{Rochdi1997, Barboteu2003}, the models were taken with finite time interval and with no initial gap. Further more, in the second work \cite{Barboteu2003}, the body's material is considered elastic. In this paper we consider the last mentioned case but with different contact conditions and other constitutive law. Namely, we took the time interval to be $\left[0, \infty\right)$ and an initial gap function that could be different from zero. 
	
	To simplify notations, calculations and without loss of generality, we considered the contact model in \textbf{Figure}\ref{figure2} which is a body colliding with itself rather than the model in \textbf{Figure}\ref{figure3} which represent the contact of a two disjoint bodies with the same material. We note that the existence and uniqueness properties of the two models are the same. The body's material is considered viscoplastic which defined by the law:
	\begin{equation}
		\dot{\bm{\sigma}}(t) = \mathcal{E} \bm{\varepsilon}(\dot{\bm{u}}(t)) + \mathcal{G}(\bm{\sigma}(t), \bm{\varepsilon}(\bm{u}(t))),    {\ \ \ \ \ } t\geq0.
	\end{equation}
		
	The contact is considered frictionless, with no penetration between the contact surfaces $\Gamma^a_c$ and $\Gamma^b_c$. The proof in this work is inspired by \cite[Section~5]{Sofonea2022}, where the author examines a body coming into contact with a rigid foundation covered by a deformable layer. Moreover, the viscoplastic law in \cite[Section~5]{Sofonea2022} is not presented in its general form as it is in our work.

	
	
\begin{figure}[t]
	\centering
	\begin{minipage}[t]{0.33\textwidth}
		\centering
		\includegraphics[width=\textwidth,height=0.25\textheight,keepaspectratio]{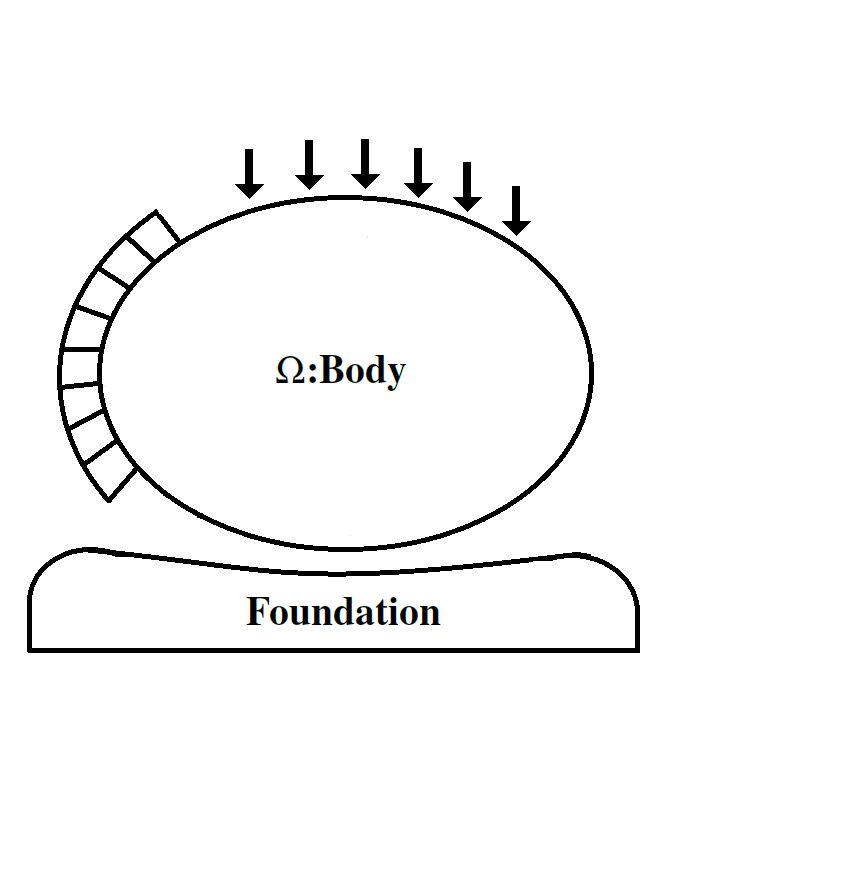}
		\caption{Contact with a rigid \\fixed foundation.}
		\label{figure1}
	\end{minipage}%
	\begin{minipage}[t]{0.33\textwidth}
		\centering
		\includegraphics[width=\textwidth,height=0.25\textheight,keepaspectratio]{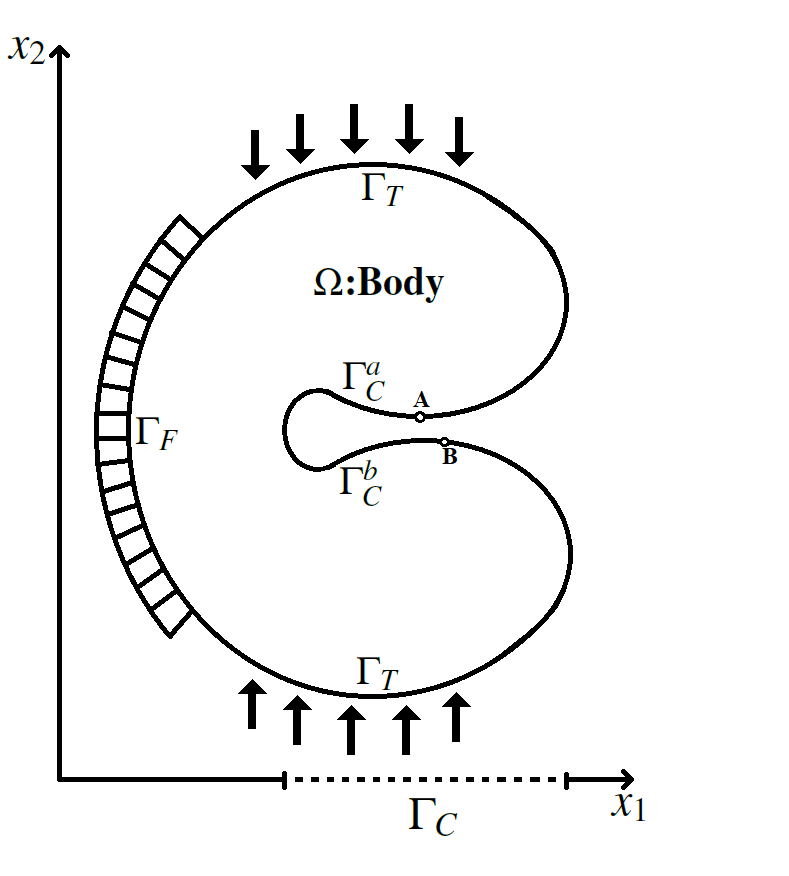}
		\caption{Contact between two \\deformable surfaces (the case\\ studied in this paper).}
		\label{figure2}
	\end{minipage}%
	\begin{minipage}[t]{0.33\textwidth}
		\centering
		\includegraphics[width=\textwidth,height=0.25\textheight,keepaspectratio]{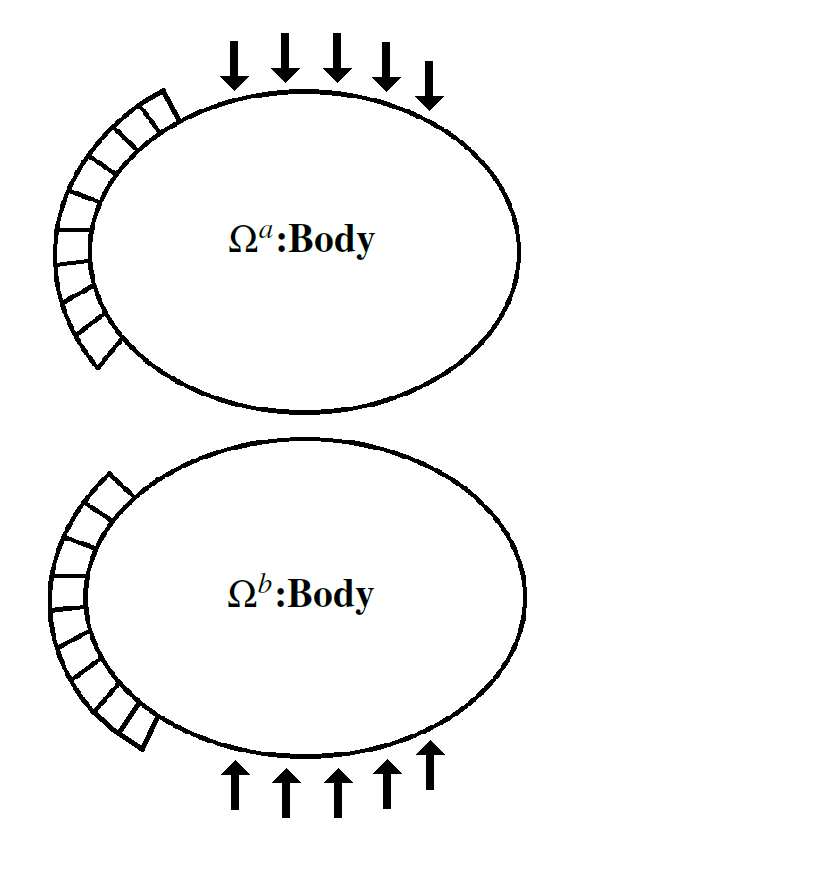}
		\caption{Contact between two \\deformable bodies.}
		\label{figure3}
	\end{minipage}
	\label{fig:contact_models}
\end{figure}
	
	The rest of the paper is divided as follows: In next section, we define the spaces, operators, and theorems that will be used in later sections. In Section 3, we give the classical modeling of the problem, then derive its weak form. In the last section, we begin by proving two useful lemmas. Using them, we state the main result, which is the existence and uniqueness theorem for the weak problem using fixed point approach.    
	
	\maketitle
	\section{Notations and Preliminaries}\label{notations and preliminaries}
	We begin by introducing the following spaces:
	\begin{equation}
		V = \left\{ v = (v_i) \in H^1(\Omega)^d \mid v_i = 0 \text{ on } \Gamma_i \text{ for } i = 1, \dots, d; \ d \in \{2,3\} \right\},
	\end{equation}	
	which represents the space of possible displacements. The following is a subset of $V$ that represents the set of admissible displacements:
	\begin{equation}
		K = \{ v \in V : v_{n}^R - g \leq 0 \text{ \ a.e. on } \Gamma_C \},
		\label{subset K}
	\end{equation}
	
	 The constraint $v_{n}^R - g \leq 0$ represents the non-penetration condition of the upper side $\Gamma_C^a$ relative to the lower side $\Gamma_C^b$ in the contact area of the body shown in \textbf{Figure} \ref{figure2}. Additionally, points A and B have the same $x_1$ coordinates at time $t\geq0$. For further details, see \cite[Chapter~6.8.]{KikuchiOden1988}. The set $K$ is a nonempty closed convex subset of $V$.  
	
	\begin{equation}
		Q = \left\{ \sigma \in \mathbb{S}^d \mid \sigma_{ij} \in L^2(\Omega); \, i,j = 1, \dots, d \right\},
	\end{equation}	
	
	where $\mathbb{S}^d$ is the space of second-order symmetric tensors on $\mathbb{R}^d$ and $0_{\mathbb{S}^d}$ denotes its additive identity element. We denote by ``$\cdot$" and $\|\cdot\|$ the inner product and the Euclidean norm, respectively, on the spaces $\mathbb{R}^d$ and $\mathbb{S}^d$,  The spaces \( V \) and \( Q \) become real Hilbert spaces when endowed with the inner products:
	
	\begin{equation}
		(u, v)_V = \int_{\Omega} \varepsilon(u) \cdot \varepsilon(v) \,dx, \quad \forall u,v \in V,
	\end{equation}
	where $\varepsilon$ represents the strain operator defined by:
	\begin{equation}
		\varepsilon(v) = (\varepsilon_{ij}(v)); \quad \varepsilon_{ij}(v) = \frac{1}{2} (v_{i,j} + v_{j,i}), \quad \forall v \in V,
	\end{equation}
	the notation $v_{i,j}=\partial v_{i}/\partial x_{j}$ denotes the partial derivative of $v_{i}$ with respect to the spatial variable $x_{j}$. For $\sigma,\tau \in Q$:
	
	\begin{equation}
		(\sigma, \tau)_Q = \int_{\Omega} \sigma \cdot \tau \,dx.
	\end{equation}
	
	 We denote by $C(\mathbb{R}_+, V)$ and $C(\mathbb{R}_+, Q)$ the spaces of  continuous functions mapping $\mathbb{R}_+$ to $V$ and $\mathbb{R}_+$ to $Q$, respectively.
	We define the following real Banach space $Q_{\infty}$ as: 
	\begin{equation}
		Q_{\infty} = \left\{ A = (a_{ijkl}) \mid a_{ijkl} = a_{jikl} = a_{klij} \in L^\infty(\Omega); \, 1 \leq i,j,k,l \leq d \right\},
	\end{equation}
	equipped with the norm:	
	\begin{equation}
		\|A\|_{Q_{\infty}} = \max_{0 \leq i,j,k,l \leq d} \|a_{ijkl}\|_{L^{\infty}(\Omega)}.
	\end{equation}
	
	In the space $Q_{\infty}$, the following inequality holds:
	
	\begin{equation}
		\|A \tau\|_Q \leq d \|A\|_{Q_{\infty}} \|\tau\|_Q, \quad \forall A \in Q_{\infty}, \, \tau \in Q.
		\label{Q infinity property}
	\end{equation}

	Now we state assumptions that hold throughout the rest of paper and define the operators that will be used in the modeling and weak formulation of our problem. We begin by assuming that the material constitutive functions $\mathcal{E}$ and $\mathcal{G}$ satisfy:   
	
	\begin{equation}
		\left\{
		\begin{array}{ll}
			\text{(a)} \quad \mathcal{E} \in Q_{\infty}. \\
			\text{(b)} \quad \text{There exists a constant } m_{\mathcal{E}} > 0 \text{ such that} \\
			\quad \quad \mathcal{E} \tau \cdot \tau \geq m_{\mathcal{E}} \|\tau\|^2 \quad \text{for all } \tau \in \mathbb{S}^d, \text{ a.e. in } \Omega.
		\end{array}
		\right.
		\label{elastic tensor}
	\end{equation}
	
	\begin{equation}
		\label{inelastic operator}
		\left\{
		\begin{array}{ll}
			\text{(a)} \quad \mathcal{G} : \Omega \times \mathbb{S}^d \times \mathbb{S}^d \to \mathbb{S}^d. \\
			\text{(b)} \quad \text{There exists a constant } L_{\mathcal{G}} > 0 \text{ such that} \\
			\quad \quad \|\mathcal{G}(x, \sigma_1, \varepsilon_1) - \mathcal{G}(x, \sigma_2, \varepsilon_2)\|
			  \leq L_{\mathcal{G}} \left( \|\sigma_1 - \sigma_2\| + \|\varepsilon_1 - \varepsilon_2\| \right) \\
			\quad \quad \forall \sigma_1, \sigma_2, \varepsilon_1, \varepsilon_2 \in \mathbb{S}^d, \text{ a.e. in } \Omega. \\
			\text{(c)} \quad \text{The mapping } x \mapsto \mathcal{G}(x, \sigma, \varepsilon) \text{ is measurable,} 
			\, \forall \sigma, \varepsilon \in \mathbb{S}^d. \\
			\text{(d)} \quad \text{The mapping } x \mapsto \mathcal{G}(x, 0_{\mathbb{S}^d}, 0_{\mathbb{S}^d}) \text{ belongs to } Q.
		\end{array}
		\right.
	\end{equation}
	   
	We endow the space $V$ with the inner product: 	   
	\begin{equation}
		( u, v )_{\mathcal{E}} = \int_{\Omega} \mathcal{E} \varepsilon(u) \cdot \varepsilon(v) \,dx, \quad u,v \in V,
	\end{equation}
	and $\|\cdot\|_{\mathcal{E}}$ is the corresponding norm. The norms $\|\cdot\|_{\mathcal{E}}$ and $\|\cdot\|_V$ are equivalent, since \eqref{Q infinity property} and \eqref{elastic tensor}(b) give us:
	\begin{equation}
		\sqrt{m_{\mathcal{E}}} \|v\|_{V} \leq \|v\|_{\mathcal{E}} 
		\leq \sqrt{d\|\mathcal{E}\|_{Q_{\infty}}} \|v\|_{V}, 
		\quad \forall v \in V.
	\end{equation}

	The following operators will be used in Sections~\ref{modeling and variational formulation} and \ref{existence and uniqueness results section}: $A: V \to V$, $\Lambda : C(\mathbb{R}_+; Q) \times C(\mathbb{R}_+; Q) \to C(\mathbb{R}_+; Q)$, and the function $f : \mathbb{R}_+ \to V$, defined by:
	
	\begin{equation}
		(A u, v)_V = (u, v)_\mathcal{E}, \quad \forall u,v \in V.
	\label{operator A}
	\end{equation}
	
	\begin{equation}
		\Lambda(\sigma, \tau)(t) = \int_0^t \mathcal{G}({\sigma}(s), \tau(s))ds 
		+ \sigma_0 - \mathcal{E} \varepsilon (u_0), 
		\quad \forall \sigma, \tau \in C(\mathbb{R}_+; Q), \, t \in \mathbb{R}_+.
	\end{equation}
	
	\begin{equation}
		(f(t), v)_V = \int_{\Omega} f_0 (t) \cdot v \, dx 
		+ \int_{\Gamma_T} f_1 (t) \cdot v \, d\Gamma, 
		\quad \forall v \in V, \, t \in \mathbb{R}_+,
	\end{equation}
	where $f_0$	is the body force density and $f_1$	is the surface traction density. The primal variational formulation for a Signorini-type contact problem is:
	\begin{equation} \label{primal variational formulation}
	    \forall t\geq0, \text{ \ \ } u(t) \in K \quad \text{and} \quad	\int_{\Omega} \sigma(t) \cdot \varepsilon(v-u(t)) \,dx \geq (f(t), v - u(t))_V,\quad \forall v\in K.
	\end{equation}
	For further details on inequality \eqref{primal variational formulation}, we refer the reader to \cite[Chapter~2.5.]{KikuchiOden1988}. We conclude this section by recalling some results that will be used in establishing our main results in the final section.
	\begin{proposition}\cite[Proposition~31]{brezis1968} \label{Brezis proposition}
		Let $H$ be a Hilbert space and $X$ a closed convex subset of $H$.  
		Let $A$ be a Lipschitz mapping from $X$ into $H$ satisfying  
		\[
		(Ax - Ay, x - y) \geq c \|x - y\|^2, \quad \forall x, y \in X \text{ with } c > 0.
		\]
		Let $\varphi$ be a lower semicontinuous convex function from $X$ into $]-\infty, +\infty]$, not identically equal to $+\infty$.  
		Then, for any $f \in H$ there exists a unique $u \in X$ such that  
		\begin{equation} \label{ineq4.3}
			(f, v - u) - (Au, v - u) \leq \varphi(v) - \varphi(u), \quad \forall v \in X.
		\end{equation}	
		Moreover, inequality \eqref{ineq4.3} can be solved using a successive approximations method.
	\end{proposition}
	
	\begin{definition}
		An operator \( S: C(\mathbb{R}_+;X) \to C(\mathbb{R}_+;X) \) is called an \textit{Almost History-Dependent} operator if and only if for any \( m \in \mathbb{N} \) there exist \( l_m \in [0,1) \) and \( L_m > 0 \) such that
		\begin{equation}
			\| S u_1 (t) - S u_2 (t) \|_X \leq l_m \| u_1 (t) - u_2 (t) \|_X 
			+ L_m \int_0^t \| u_1 (s) - u_2 (s) \|_X ds
		\end{equation}
		for all \( u_1, u_2 \in C(\mathbb{R}_+;X) \), \( t \in [0,m] \).
	\end{definition}
		
	\begin{theorem}\cite[Theorem~31]{SofoneaAHD}\label{ahd theorem}
		Let \( (X, \|\cdot\|_X) \) be a Banach space and \( S : C(\mathbb{R}_+;X) \to C(\mathbb{R}_+;X) \) be an Almost History-Dependent operator. Then \( S \) has a unique fixed point \( \eta \in C(\mathbb{R}_+;X) \) i.e., \quad $S\eta (t)= \eta(t) \quad \forall t \in \mathbb{R}_+$ .
	\end{theorem}

	\maketitle
	\section{Modeling and Variational Formulation}\label{modeling and variational formulation}

		We consider our viscoplastic body as a set of particles occupying the domain $\overline{\Omega} \subset \mathbb{R}^2$ in its reference configuration, as shown in \textbf{Figure} \ref{figure2}. Here, $\Omega$ is an open, bounded, and connected set such that $\overline{\Omega} = \Omega \cup \Gamma$. We assume that the boundary $\Gamma$ is divided into three mutually disjoint parts: $\Gamma_F$, $\Gamma_T$, and $\Gamma_C^a \cup \Gamma_C^b$. The body is fixed along $\Gamma_F$ (with $meas(\Gamma_F) > 0$), meaning that the displacement satisfies $u = 0$ on $\Gamma_F \times \mathbb{R}+$. The body is subjected to a body forces $f_0$ acting on $\Omega \times \mathbb{R}+$ and to a surface traction $f_1$ applied on $\Gamma_T \times \mathbb{R}_+$. The part $\Gamma_C^a \cup \Gamma_C^b$ contains the area that comes into contact during the time interval $[0, +\infty)$, where the contact is governed by the Signorini conditions with an initial gap function $g$, and $\Gamma_C$ is its projection on the $x_1$ axis. In the three-dimensional setting, the present two-dimensional model represents a cross-section of the three-dimensional body. We assume that the process is quasistatic, the stress is in equilibrium, and that the contact is frictionless. Accordingly, the classical formulation of the problem under consideration is as follows:
		
		\begin{problem}\label{classic problem}
		Find a displacement field $u:\Omega\times[0,+\infty)\to\mathbb{R}^d$ and a stress tensor $\sigma:\Omega\times[0,+\infty)\to\mathbb{S}^d$	such that $\forall t\geq0$, the following holds:		
			\begin{align}				
				\dot{\bm{\sigma}}(t) &= \mathcal{E} \bm{\varepsilon}(\dot{\bm{u}}(t)) + \mathcal{G}(\bm{\sigma}(t), \bm{\varepsilon}(\bm{u}(t))) \text{ in } \Omega, \quad  \label{material constitutive law}\\
				\text{Div} \, \sigma(t) + f_0(t) &= 0 \text{ in } \Omega, \quad \label{equilibrium} \\
				u(t) &= 0 \text{ on } \Gamma_F,   \\
				\sigma(t)n_T &= f_1(t) \text{ on } \Gamma_T,   \\
				u_{n}^R(t) - g &\leq 0 \text{ on } \Gamma_C, \label{signorini1}\\
				(u_{n}^R(t) - g)\sigma_n(t) &=0 \text{ on } \Gamma_C, \label{signorini2}\\
				\sigma_n(t) &\leq 0 \text{ on } \Gamma_C, \label{signorini3}\\
				\sigma_\tau(t) &= 0 \text{ on } \Gamma_C.\label{signorini4}
			\end{align}
			\end{problem}

	In \ref{equilibrium}, $\text{Div}$ represents the divergence operator of the stress tensor $\sigma$, defined as $\text{Div} \, \sigma(t) = \big(\sum_{j=1}^{d} \frac{\partial \sigma(t)_{i,j}}{\partial x_{j}}\big)_{i=1,...,d}$. $n_T$ denotes the normal outward unit vector on $\Gamma_T$. \ref{signorini1} $\mathord{-}$ \ref{signorini4} represents the Signorini-type contact conditions, detailed in \cite[Chapter~6.8.]{KikuchiOden1988}. $\sigma_n$ and $\sigma_\tau$ represent the normal and tangential stress, respectively. Besides the definitions and assumptions on the data in Section \ref{notations and preliminaries}, we assume that the body forces, surface tractions, and initial conditions satisfy the following: 
	\begin{equation}
		f_0 \in C(\mathbb{R}_+; L^2(\Omega)^d).
		\label{body forces}
	\end{equation}
	\begin{equation}
		f_1 \in C(\mathbb{R}_+; L^2(\Gamma_T)^d).
		\label{surface traction}
	\end{equation}
	\begin{equation}
		u(0)=u_0 \in V, \quad \sigma(0)=\sigma_0 \in Q.
		\label{initial data}
	\end{equation}
	\maketitle
	
	We now proceed to construct a variational formulation of our problem. Assuming that $(u,\sigma)$ is a smooth solution to the problem \ref{classic problem}, let $v \in K$ and $t\geq0$.\\
	By integrating \ref{material constitutive law} over time, we obtain:  
	\begin{equation}\label{integrated stress}
		\sigma(t) = \mathcal{E} \varepsilon (u(t)) + \Lambda(\sigma, \varepsilon(u))(t).
	\end{equation}
	Substituting this into the primal variational formulation \ref{primal variational formulation}, we obtain:
	\begin{equation}\label{primal variational formulation of our problem}
		\int_{\Omega} \mathcal{E} \varepsilon(u(t)) \cdot \varepsilon(v-u(t))+\int_{\Omega} \Lambda(\sigma, \varepsilon(u))(t) \cdot \varepsilon(v-u(t)) \,dx \geq (f(t), v - u(t))_V.
	\end{equation} 
	Using \ref{integrated stress} and \ref{primal variational formulation of our problem}, we conclude the following weak formulation corresponding to the classical problem \ref{classic problem}:
	
	\begin{problem}\label{problempw}
		Find $u \in C(\mathbb{R}_+, V)$, $\sigma \in C(\mathbb{R}_+, Q)$, and $\eta \in C(\mathbb{R}_+, Q)$ such that $\forall t \in \mathbb{R}_+$, we have:
		\begin{equation*}
			\left \{
			\begin{array}{ll}
				\sigma(t) = \mathcal{E} \varepsilon (u(t)) + \eta (t) \\
				u(t) \in K,\quad (u(t), v - u(t))_\mathcal{E} + (\eta(t)), \varepsilon (v - u(t)))_Q \geq (f(t), v - u(t))_V, \quad \forall v \in K. \\
				\eta(t) =  \Lambda(\sigma, \varepsilon(u))(t)
			\end{array}
			\right. 
		\end{equation*}	
			
	\end{problem}
	
	\section{Existence and Uniqueness Results} \label{existence and uniqueness results section} 
	
	We first prove the following two lemmas corresponding to our weak form.
	
	\begin{lemma} \label{lemma1}
	For each $\eta \in C(\mathbb{R}_+, Q)$, there exists a unique element $u_\eta \in C(\mathbb{R}_+, V)$ such that, $\forall t \in \mathbb{R}_+, \quad u_\eta(t) \in K$ and
	\begin{equation}
		 \quad (A u_\eta(t), v - u_\eta(t))_V + (\eta(t)), \varepsilon (v - u_\eta(t)))_Q \geq \quad (f(t), v - u_\eta(t))_V, \quad \forall v \in K.
	\label{lemma inequality}
	\end{equation}
	\end{lemma}
   	
   	\begin{proof} 
 	Let $\eta \in C(\mathbb{R}_+, Q)$.
 	
   	\textit{Existence and Uniqueness:} 
  	Let $t \in \mathbb{R}_+$, and define $\varphi: K \to \mathbb{R}$ by $\varphi(v) = (\eta(t), \varepsilon(v))_Q$, $\forall v \in K$. Now, \eqref{elastic tensor} implies that the operator $A$ defined in \eqref{operator A} is strongly monotone and Lipschitz continuous. Moreover, $K$, defined in \eqref{subset K}, is a closed convex subset of $V$, and \eqref{body forces} and \eqref{surface traction} imply that $f \in C(\mathbb{R}_+, V)$. Additionally, $\varphi$ is a convex lower semicontinuous function. By Proposition \ref{Brezis proposition}, there exists a unique element $u_\eta(t) \in K$ that satisfies \eqref{lemma inequality}, for all $t \in \mathbb{R}_+$.  
   	
   	Next, we show that $u_\eta: \mathbb{R}_+ \to V$ is continuous.  
   	          
   	\textit{Continuity:} Let $t' \in \mathbb{R}_+$ and $\varepsilon > 0$. Since $f$ and $\eta$ are continuous mappings on $\mathbb{R}_+$, we have:
   	
   	\begin{equation*}
   		\exists \, \delta_1 > 0 \text{ such that, for } t'' \in \mathbb{R}_+ \text{ and } |t'' - t'| < \delta_1, \text{ we have } \|f(t'') - f(t')\|_V < \frac{m \varepsilon}{2}
   	\end{equation*}
   	
   	\begin{equation*}
   		\exists \, \delta_2 > 0 \text{ such that, for } t'' \in \mathbb{R}_+ \text{ and } |t'' - t'| < \delta_2, \text{ we have } \|\eta(t'') - \eta(t')\|_Q < \frac{m \varepsilon}{2}
   	\end{equation*}
   	
   	Taking $ \delta = \min \{ \delta_1, \delta_2 \}$ and $t'' \in \mathbb{R}_+$ such that $0 < |t'' - t'| < \delta$, and using the previous existence and uniqueness result, we obtain:
	
	\begin{equation*}
		\exists \, u_{\eta}(t') \in K, \quad ( A u_{\eta}(t'), v - u_{\eta}(t') )_V + (\eta(t'), \varepsilon(v - u_{\eta}(t')))_Q \geq (f(t'), v - u_{\eta}(t'))_V, \quad \forall v \in K
	\end{equation*}
	   	
   	\begin{equation*}
   		\exists \, u_{\eta}(t'') \in K, \quad ( A u_{\eta}(t''), v - u_{\eta}(t'') )_V + (\eta(t''), \varepsilon(v - u_{\eta}(t'')))_Q \geq (f(t''), v - u_{\eta}(t''))_V, \quad \forall v \in K
   	\end{equation*}

   	choosing $v = u_{\eta}(t'')$ in the first inequality and $v = u_{\eta}(t')$ in the second, then adding them together, we obtain:
   	
   	\begin{equation*}
   		( A u_{\eta}(t') - A u_{\eta}(t''), u_{\eta}(t') - u_{\eta}(t'') )_V \leq ( \eta(t'') - \eta(t'), \varepsilon(u_{\eta}(t') - u_{\eta}(t'')) )_Q + ( f(t') - f(t''), u_{\eta}(t') - u_{\eta}(t''))_V
   	\end{equation*}
   	
   	Since $A$ is a strongly monotone operator, applying the Cauchy-Schwarz inequality and using the fact that $\|\varepsilon(\cdot)\|_Q = \|\cdot\|_V$, we obtain:
   	
   	\begin{equation*}
   		m \| u_{\eta}(t') - u_{\eta}(t'') \|_V^2 \leq \| \eta(t') - \eta(t'') |_Q \| u_{\eta}(t') - u_{\eta}(t'') \|_V + \| f(t') - f(t'') |_V \| u_{\eta}(t') - u_{\eta}(t'') \|_V
   	\end{equation*}
   	
   	By the continuity of $f$ and $\eta$ at $t'$, it follows that:
   	
   	\begin{equation*}
   		\| u_{\eta}(t') - u_{\eta}(t'') \|_V \leq \frac{1}{m} \| \eta(t') - \eta(t'') \|_Q + \frac{1}{m} \| f(t') - f(t'') \|_V < \varepsilon
   	\end{equation*}
   	
   	Thus, the function $ u_{\eta} : \mathbb{R}_+ \to V$ is continuous at every point $t' \in \mathbb{R}_+$, i.e., $u_{\eta} \in C(\mathbb{R}_+, V)$.
 
   	\end{proof}
	
	\begin{lemma}\label{estimation lemma}
		For each $\eta_i \in C(\mathbb{R}_+, Q)$, $i \in \{1,2\}$, there exists a corresponding unique pair $(\sigma_{\eta_i}, u_{\eta_i}) \in C(\mathbb{R}_+, Q) \times C(\mathbb{R}_+, V)$ that satisfies, for all $t \in \mathbb{R}_+$:  
		
		\begin{equation}
			\sigma_{\eta_i}(t) = \mathcal{E} \varepsilon (u_{\eta_i}(t)) + \eta_i (t)
			\label{lemma2,1}
		\end{equation}		
		\begin{equation}
			u_{\eta_i}(t) \in K, \quad ( A u_{\eta_i}(t), v - u_{\eta_i}(t) )_V + (\eta_i(t), \varepsilon(v - u_{\eta_i}(t)))_Q \geq ( f(t), v - u_{\eta_i}(t) )_V, \quad \forall v \in K.
			\label{lemma2,2}
		\end{equation}
		
		Moreover:	
		\begin{equation}
			\| \sigma_{\eta_1}(t) - \sigma_{\eta_2}(t) \|_Q + \| \varepsilon (u_{\eta_1}(t)) - \varepsilon (u_{\eta_2}(t)) \|_Q \leq C \| \eta_1 (t) - \eta_2 (t) \|_Q.
			\label{lemma2,3}
		\end{equation}
	\end{lemma}
	
	\begin{proof}
		Let $\eta_i \in C(\mathbb{R}_+, Q)$, $i \in \{1,2\}$. Lemma \ref{lemma1} implies that for each $\eta_i$, there exists a unique pair $(\sigma_{\eta_i}, u_{\eta_i})$ in $C(\mathbb{R}_+, Q) \times C(\mathbb{R}_+, V)$ that satisfies \eqref{lemma2,1} and \eqref{lemma2,2}, respectively.  
		
		\textit{Proving \eqref{lemma2,3}:}  
		In \eqref{lemma2,2}, by taking $v = u_{\eta_2}$ for $\eta_1$ and $v = u_{\eta_1}$ for $\eta_2$, we obtain:
		
		\begin{equation*}
			( A u_{\eta_1}(t) - A u_{\eta_2}(t), u_{\eta_2}(t) - u_{\eta_1}(t) )_V + (\eta_1 (t) - \eta_2 (t), \varepsilon (u_{\eta_2}(t)) - \varepsilon (u_{\eta_1}(t)) )_Q \geq 0, \quad \forall t \in \mathbb{R}_+.
		\end{equation*}
		
		Since $A$ is strongly monotone, $\|\varepsilon(\cdot)\|_Q = \|\cdot\|_V$, and using the Cauchy-Schwarz inequality, we obtain:
		
		\begin{equation}
			\| \varepsilon (u_{\eta_1}(t)) - \varepsilon (u_{\eta_2}(t)) \|_Q \leq C_1 \| \eta_1 (t) - \eta_2 (t) \|_Q, \quad \forall t \in \mathbb{R}_+.
			\label{partial result 1}
		\end{equation}
		
		From \eqref{lemma2,1}, we have:	
		\begin{equation*}
			\| \sigma_{\eta_1} (t) - \sigma_{\eta_2} (t) \|_Q \leq \| \mathcal{E} \varepsilon (u_{\eta_1}(t) - u_{\eta_2}(t)) \|_Q + \| \eta_1 (t) - \eta_2 (t) \|_Q, \quad \forall t \in \mathbb{R}_+.
		\end{equation*}
		
		Using \eqref{Q infinity property} and \eqref{partial result 1}, we obtain:
		
		\begin{equation}
			\| \sigma_{\eta_1} (t) - \sigma_{\eta_2} (t) \|_Q \leq C_2 \| \eta_1 (t) - \eta_2 (t) \|_Q, \quad \forall t \in \mathbb{R}_+.
			\label{partial result 2}
		\end{equation}
		
		By adding \eqref{partial result 1} and \eqref{partial result 2}, we obtain estimation \ref{lemma2,3}.
	\end{proof}
	
	We now state the unique solvability result of the weak problem \ref{problempw}. 
	
	\begin{theorem}
		Assume that the data satisfy \eqref{subset K}, \eqref{elastic tensor}, \eqref{inelastic operator}, \eqref{body forces}, \eqref{surface traction}, and \eqref{initial data}. Then, the problem \ref{problempw} has a unique solution.
	\end{theorem}
	\begin{proof}
		Let $S: C(\mathbb{R}_+, Q) \to C(\mathbb{R}_+, Q)$ be an operator defined as follows: for $\eta \in C(\mathbb{R}_+, Q)$, we have:
		\begin{equation*}
			S\eta(t) = \Lambda \big( \sigma_{\eta}(t), \varepsilon (u_{\eta}(t)) \big) = \int_0^t \mathcal{G} (\sigma_{\eta} (s), \varepsilon (u_{\eta} (s))) ds + \sigma_0 - \mathcal{E} \varepsilon(u_0)
		\end{equation*}
		where the pair $(\sigma_\eta, u_{\eta})$ satisfies \eqref{lemma2,1} \eqref{lemma2,2} in Lemma \ref{estimation lemma} for the element $\eta$.
		
		For any two elements $\eta_1, \eta_2 \in C(\mathbb{R}_+, Q)$, we have:
		\begin{equation}\label{operator S estimation}
		\begin{array}{rl}
			\| S\eta_1 (t) - S\eta_2 (t) \|_Q &= \left\| \int_0^t \mathcal{G} (\sigma_{\eta_1} (s), \varepsilon (u_{\eta_1} (s))) - \mathcal{G} (\sigma_{\eta_2} (s), \varepsilon (u_{\eta_2} (s))) ds \right\|_Q \\		
			&\leq \int_0^t \| \mathcal{G} (\sigma_{\eta_1} (s), \varepsilon (u_{\eta_1} (s))) - \mathcal{G} (\sigma_{\eta_2} (s), \varepsilon (u_{\eta_2} (s))) \|_Q ds.
		\end{array}
		\end{equation}
		On the other hand, using \ref{inelastic operator}(b), the fact that $\|.\|_Q =\sqrt{\int_\Omega \|.\|^2 dx}$, and the Cauchy-Schwarz inequality on the functions $\alpha, \beta \in L^2(\Omega)$ defined by:
		\begin{equation*}
			\alpha = \| \sigma_{\eta_1} (s) - \sigma_{\eta_2} (s) \|, \quad \beta = \| \varepsilon (u_{\eta_1} (s)) - \varepsilon (u_{\eta_2} (s)) \|,
		\end{equation*}
		we obtain:
		\begin{equation*}
			\begin{array}{rl}
				\| \mathcal{G} (\sigma_{\eta_1} (s), \varepsilon (u_{\eta_1} (s))) - \mathcal{G} (\sigma_{\eta_2} (s), \varepsilon (u_{\eta_2} (s))) \|_Q^2
				& \leq L_\mathcal{G}^2 \int_\Omega \big( \| \sigma_{\eta_1} (s) - \sigma_{\eta_2} (s) \| + \| \varepsilon (u_{\eta_1} (s)) - \varepsilon (u_{\eta_2} (s)) \| \big)^2 dx \\
				
				& \leq L_\mathcal{G}^2  (\| \sigma_{\eta_1} (s) - \sigma_{\eta_2} (s) \|^2_Q + \| \varepsilon (u_{\eta_1} (s)) - \varepsilon (u_{\eta_2} (s)) \|^2_Q \\
				&\quad + 2\int_\Omega \| \sigma_{\eta_1} (s) - \sigma_{\eta_2} (s) \|\| \varepsilon (u_{\eta_1} (s)) - \varepsilon (u_{\eta_2} (s)) \| dx) \\
				 
				& \leq L_\mathcal{G}^2  (\| \sigma_{\eta_1} (s) - \sigma_{\eta_2} (s) \|^2_Q + \| \varepsilon (u_{\eta_1} (s)) - \varepsilon (u_{\eta_2} (s)) \|^2_Q \\  
				&\quad + 2\| \sigma_{\eta_1} (s) - \sigma_{\eta_2} (s) \|_Q\| \varepsilon (u_{\eta_1} (s)) - \varepsilon (u_{\eta_2} (s)) \|_Q)  \\
				
				& \leq L_\mathcal{G}^2 \big( \| \sigma_{\eta_1} (s) - \sigma_{\eta_2} (s) \|_Q + \| \varepsilon (u_{\eta_1} (s)) - \varepsilon (u_{\eta_2} (s)) \|_Q \big)^2.
			\end{array}
		\end{equation*}
		
		Substituting the square root of this into \eqref{operator S estimation} and using \eqref{lemma2,3} in Lemma \ref{estimation lemma}, we obtain:
		\begin{equation*}
			\| S_{\eta_1}(t) - S_{\eta_2}(t) \|_Q \leq L_m \int_0^t \| \eta_1 (s) - \eta_2 (s) \|_Q ds, \quad \forall t \in [0, m],
		\end{equation*}
		
		where $L_m = L_g C$, and $m \in \mathbb{N}$. Hence, $S$ is an Almost History-Dependent operator. Theorem \ref{ahd theorem} implies that $S$ has a unique fixed point $\eta_S \in C(\mathbb{R}_+, Q)$. Thus, $(\eta_S, \sigma_{\eta_S}, u_{\eta_S})$ is the unique solution to problem \ref{problempw}.
		
	\end{proof}

\end{document}